\documentclass{amsart}
\usepackage{verbatim}
\usepackage[textsize=scriptsize]{todonotes}
\usepackage{etoolbox}
\newtoggle{draft}
\newtoggle{final}

\usepackage{tikz}
\usetikzlibrary{matrix,arrows}
\usepackage{tikz-cd}

\togglefalse{draft}

\usepackage{etoolbox}
\iftoggle{draft} {
\usepackage[margin=1.5in]{geometry}
}{ 
\usepackage[margin=1in]{geometry}
\setlength{\marginparwidth}{0.75in}
}
\usepackage{amsmath}
\usepackage{amssymb}
\usepackage{amsthm}
\usepackage{amscd}
\usepackage{enumerate}
\usepackage[pdfusetitle,unicode,hidelinks]{hyperref}
\usepackage{bbm}
\usepackage{etoolbox}

\usepackage[utf8]{inputenc}

\newcommand{\tomemail}{\href{mailto:tom.bachmann@zoho.com}{tom.bachmann@zoho.com}}

\newtheorem{proposition}{Proposition}
\newtheorem{corollary}[proposition]{Corollary}
\newtheorem{lemma}[proposition]{Lemma}
\newtheorem{theorem}[proposition]{Theorem}
\newtheorem{conjecture}[proposition]{Conjecture}
\newtheorem*{conjecture*}{Conjecture}
\newtheorem*{theorem*}{Theorem}
\newtheorem*{corollary*}{Corollary}
\newtheorem*{proposition*}{Proposition}
\newtheorem*{lemma*}{Lemma}
\theoremstyle{definition}
\newtheorem{definition}[proposition]{Definition}

\newtheorem*{definition*}{Definition}
\newtheorem*{construction*}{Construction}
\theoremstyle{remark}
\newtheorem{remark}[proposition]{Remark}
\newtheorem*{remark*}{Remark}

\newtheorem{example}[proposition]{Example}
\newtheorem*{example*}{Example}

\newcommand{\id}{\operatorname{id}}
\newcommand{\Z}{\mathbb{Z}}

\newcommand{\N}{\mathbb{N}}

\let\scr=\mathcal
\let\bb=\mathbb
\newcommand{\Gm}{{\mathbb{G}_m}}
\newcommand{\Gmp}[1]{{\mathbb{G}_m^{\wedge #1}}}
\def\A{\bb A}

\newcommand{\1}{\mathbbm{1}}

\newcommand{\eff}{{\text{eff}}}
\newcommand{\veff}{{\text{veff}}}

\newcommand{\SH}{\mathcal{SH}}

\DeclareMathOperator*{\colim}{colim}

\let\lim=\relax
\DeclareMathOperator*{\lim}{lim}

\def\PSh{\mathcal{P}}

\def\Spc{\mathcal{S}\mathrm{pc}{}}
\def\Fin{\cat F\mathrm{in}}

\newcommand{\wequi}{\simeq}

\DeclareRobustCommand{\ul}{\underline}

\let\cat=\mathrm
\def\Sm{{\cat{S}\mathrm{m}}}

\def\FEt{\mathrm{FEt}{}}

\def\mot{\mathrm{mot}}

\newcommand{\fr}{\mathrm{fr}}

\def\ph{\mathord-}

\numberwithin{proposition}{section}

\iftoggle{draft} {
\newcommand{\NB}[1]{\todo[color=gray!40]{#1}}
}{ 
\newcommand{\NB}[1]{}
\renewcommand{\todo}[1]{}
}

\newcommand{\SHS}{\mathcal{SH}^{S^1}\!}
\newcommand{\Hilb}{\mathrm{Hilb}}
\newcommand{\bir}{\mathrm{bir}}
\newcommand{\cof}{\mathrm{cof}}
\newcommand{\gp}{\mathrm{gp}}
\newcommand{\Corr}{\mathrm{Corr}}
\newcommand{\FSyn}{\mathcal{FS}\mathrm{yn}}
\newcommand{\flci}{\mathrm{flci}}

\newcommand{\GL}{\mathrm{GL}}
\newcommand{\cod}{\mathrm{cod}}
\newcommand{\KGL}{\mathrm{KGL}}
\newcommand{\Gr}{\mathrm{Gr}}

\title{Voevodsky's slice conjectures via Hilbert schemes}
\date{\today}

\begin{document}

\author{Tom Bachmann}
\email{\tomemail}
\address{Department of Mathematics, Massachusetts Institute of Technology,
Cambridge, MA, USA}

\author{Elden Elmanto}
\email{\href{mailto:elmanto@math.harvard.edu}{elmanto@math.harvard.edu}}
\address{Department of Mathematics, Harvard University, Cambridge, USA}


\begin{abstract} We offer short and conceptual reproofs of some conjectures of Voevodsky's on the slice filtration. The original proofs were due to Marc Levine using  the homotopy coniveau tower. Our new proofs use very different methods, namely, recent development in motivic infinite loop space theory together with the birational geometry of Hilbert schemes.
\end{abstract}

\maketitle

\tableofcontents

\section{Introduction}

One major application of motivic homotopy theory is Voevodsky's construction of the Atiyah-Hirzebruch spectral sequence from motivic cohomology (which coincides Bloch's higher Chow groups \cite{bloch} of algebraic cycles up to reindexing \cite{voevodsky-compare}) converging to algebraic $K$-theory. While other constructions of this spectral sequence were proposed prior to motivic homotopy theory (notably \cite{friedlander-suslin,levine-techniques}), Voevodsky's approach is arguably the cleanest and most definitive --- we refer to \cite[\S 2]{levine-appreciate} for a survey.  This spectral sequence is obtained via the \emph{slice filtration} constructed in \cite{voevodsky-slice-filtration}, \cite{voevodsky293possible}, which is a functorial filtration that one associates to a motivic spectrum $E$:
\[
\cdots \rightarrow   f_qE \rightarrow f_{q-1}E \rightarrow \cdots f_0E \rightarrow \cdots E.
\]
The associated graded spectra are denoted
\[
s_qE := \mathrm{cofib}(f_{q+1}E \rightarrow f_qE),
\]
and are called the \emph{$q$-th slice} of $E$. Letting $E = \KGL$, the motivic spectrum representing algebraic $K$-theory, one obtains the desired spectral sequence.

While the construction of this filtration is formal, the identification of the spectral sequence (in other words, the associated graded) hinged on the next two conjectures.
They were stated by Voevodsky \cite{voevodsky293possible} and proved by Levine \cite{levine2008homotopy}.

\begin{conjecture} \label{conj:1} \cite[Theorem 10.5.1]{levine2008homotopy}, \cite[Conjecture 2]{voevodsky293possible} Let $k$ be a perfect field and $\1_k$ denote the motivic sphere spectrum. Then $s_0\1_k$ canonically identifies with the spectrum representing motivic cohomology.
\end{conjecture}

\begin{conjecture} \label{conj:2} \cite[Theorem 9.0.3]{levine2008homotopy}, \cite[Conjecture 3]{voevodsky293possible} Let $k$ be a perfect field. The functor
\[
\omega^{\infty}:\SH(k) \rightarrow \SH^{S^1}(k),
\]
respects the slice filtration.
\end{conjecture}

We will recall the definition of the slice filtrations on $\SH(k)$ and $\SH^{S^1}(k)$ in the main text. Voevodsky further proved that the validity of Conjecture~\ref{conj:2} ensures convergence of the resulting spectral sequence \cite[Corollary 3.4]{voevodsky293possible}, while Conjecture~\ref{conj:1} identifies the graded spectra as suspensions of the motivic cohomology spectrum \cite[Section 5]{voevodsky293possible} based on periodicity properties of the motivic spectrum representing algebraic $K$-theory. As already mentioned in the first paragraph, these motivic cohomology spectra have concrete incarnations as Bloch's higher Chow groups. In total, we obtain a strongly convergent, cohomologically-indexed spectral sequence:
\[
E^{p,q}_2  = H^{p-q}(X;\mathbb{Z}(-q)) = CH^{-q}(X, -p-q) \Rightarrow K_{-p-q}(X),
\]
whenever $X$ is a smooth scheme over a field. 

The purpose of this paper is to give an independent, short and conceptual proof of Conjecture~\ref{conj:2} and a simplification of Levine's proof of Conjecture~\ref{conj:1}, \emph{assuming motivic infinite loop space theory} \cite{EHKSY, EHKSY3, BEHKSY} as developed by the authors and Hoyois, Khan, Sosnilo and Yakerson, building on foundational work of Ananyevskiy, Druzhinin, Garkusha, Neshitov, and Panin in the seminal papers \cite{panin-framed-motive, garkusha2015homotopy, agp, gnp, DruzhininPanin}, based on unpublished ideas of Voevodsky's. In particular, these papers gave rise to motivic infinite loop space theory by computing the infinite $\mathbb{P}^1$-loop space of a variety in terms of framed correspondences. In fact a proof of Conjecture~\ref{conj:2} along these lines was already envisioned by Voevodksy in \cite{voevodsky293possible}, although our proof proceeds via rather different methods.

The proof of Conjecture~\ref{conj:2} is independent because we make no reference to Levine's proof. It is short, given the length of this paper. Finally, it is conceptual because we can reformulate both conjectures as relatively elementary statements about the birational geometry of certain Hilbert schemes. Indeed, motivic infinite loop space theory furnishes for us geometric models for the infinite loop space of the motivic sphere spectrum and, in fact, the suspension spectrum of any smooth $k$-variety as framed Hilbert schemes \cite[\S 5.1]{EHKSY}. That slices have something to do with the birational geometry of varieties is already well-known in the literature \cite{kahn2016birational, pelaez2014unstable}. In lieu of proving Conjecture \ref{conj:1} as stated, we identify $s_0(\1)$ with a certain explicit framed suspension spectrum. While this characterizes the spectrum uniquely, the relationship with higher Chow groups is not clear from this perspective. On the other hand, this simplifies the proof of Conjecture~\ref{conj:1} by replacing Levine's use of his ``reverse cycle map" with Hilbert schemes argument; see Remark~\ref{rem:cyc}.

\subsection*{Acknowledgements} We would like to acknowledge the influence of Marc Hoyois and Marc Levine on this paper and our education. 
We thank Maria Yakerson for comments on a draft.

\subsection*{Notation and conventions}
We fix a field $k$.
We make use of the categories and functors depicted in the following diagram.

\begin{equation*}
\begin{tikzcd}
\Spc(k)_* \ar[r, "\Sigma^\infty_{S^1}", bend left=10] \ar[rr, "\Sigma^\infty" swap, bend left=40] \ar[d, "F" swap, bend right=10] & \SHS(k) \ar[r, "\sigma^\infty", bend left=10] \ar[l, "\Omega^\infty_{S^1}", bend left=10] & \SH(k) \ar[l, "\omega^\infty", bend left=10] \ar[ll, "\Omega^\infty" swap, bend left=30] \ar[dll, "\Omega^\infty_\fr", bend left=50] \\
\Spc^\fr(k) \ar[u, "U" swap, bend right=10] \ar[urr, "\Sigma^\infty_\fr", bend right]
\end{tikzcd}
\end{equation*}

Here $\Spc(k)_*$ is the pointed unstable motivic $\infty$-category (see e.g. \cite[\S2.2]{bachmann-norms}), $\SHS(k)$ is the category of $S^1$-spectra, i.e. the stabilization of $\Spc(k)_*$, $\SH(k)$ is the category of motivic spectra (see e.g. \cite[\S4.1]{bachmann-norms}), and $\Spc^\fr(k)$ is the category of motivic spaces with framed transfers \cite[\S3.2]{EHKSY}.
All parallel functors in opposite directions are adjoints, the functors $\Sigma^\infty, \Sigma^\infty_{S^1}, \sigma^\infty$ are the evident infinite suspension functors, and $U$ is the evident forgetful functor.
The diagram of left adjoints (respectively right adjoints) commutes.

We freely use the language of $\infty$-categories as set out in \cite{lurie-htt,lurie-ha}.

\section{Some birational geometry of framed Hilbert schemes} \label{sec:hilbert}
For a scheme $X$ and a point $x \in X$, we write \[ \cod_X(x) = \dim(\scr O_{X,x}) \] for the codimension of $x$ in $X$ (see also \cite[Tag 02IZ]{stacks-project}).
We will use several times the following well-known ``codimension formula''.
\begin{theorem}
Let $f: X \to Y$ be a flat morphism of locally noetherian schemes and $x \in X$.
Then \[ \cod_X(x) = \cod_Y(f(x)) + \cod_{X_{f(x)}}(x). \]
\end{theorem}
\begin{proof}
Let $y=f(x)$.
Note that $\scr O_{X_y, x} = \scr O_{X,x} \otimes_{\scr O_{Y,y}} k(y)$.
The theorem is now a restatement of \cite[Corollary 14.95]{bagbook}.
\end{proof}

For the rest of this section, we assume that all schemes are locally noetherian.
\begin{definition}
Let $d \in \Z$.
An open immersion $U \hookrightarrow X$ is called \emph{$d$-birational} if whenever $x \in X$ with $\cod_X(x) \le d$, then $x \in U$.
\end{definition}
\begin{example}
If $d < 0$, the condition is vacuous.
If $d=0$, this coincides with what is usually called birational ($U$ contains all generic points of $X$).
\end{example}
\begin{remark} \label{rmk:bir-comp}
It follows from \cite[Tag 02I4]{stacks-project} that $d$-birational open immersions are stable under composition.
\end{remark}

The codimension formula tells us that being a $d$-birational open immersion is fpqc local (on the base).

\begin{corollary} \label{corr:d-bir-bc}
Let $\alpha: U \to S$ be arbitrary and $p: Y \to S$ flat.
\begin{enumerate}
\item If $\alpha$ is a $d$-birational open immersion, then so is the base change $\alpha_Y: U_Y \hookrightarrow Y$.
\item If $p$ is surjective and $\alpha_U$ is a $d$-birational open immersion, then so is $\alpha$.
\end{enumerate}
\end{corollary}
\begin{proof}
(1) Since open immersions are stable under base change, it is enough to show that for $y \in Y$ with $\cod_Y(y) \le d$ we have $y \in U_Y$.
By the codimension formula we have $\cod_Y(y) \ge \cod_S(p(y))$, whence $p(y) \in U$ by $d$-birationality of $\alpha$.
Thus $y \in U_Y$, as needed.

(2) $\alpha$ is an open immersion by faithfully flat descent \cite[Tag 02L3]{stacks-project}.
Let $s \in S$ with $\cod_S(s) \le d$.
Let $y \in Y_s$ be a generic point, so that $\cod_{Y_s}(y) = 0$.
Then by the codimension formula we have $\cod_Y(y) = \cod_S(s) \le d$, so that $y \in U_Y$.
This implies $s \in U$, as needed.
\end{proof}

\begin{corollary} \label{cor:d-bir-prod}
If $U \hookrightarrow X$ and $V \hookrightarrow Y$ are $d$-birational open immersions of flat $S$-schemes, then so is $U \times_S V \hookrightarrow X \times_S Y$.
\end{corollary}
\begin{proof}
We have the factorization \[ U \times_S V \hookrightarrow U \times_S Y \hookrightarrow X \times_S Y \] in which both maps are $d$-birational open immersions by Corollary \ref{corr:d-bir-bc}(1), and hence so is the composite by Remark \ref{rmk:bir-comp}.
\end{proof}

The next lemma furnishes a fiberwise criterion for being $d$-birational.

\begin{lemma} \label{lemm:d-birational-fiberwise}
Let $\alpha: U \hookrightarrow X$ be an open immersion of flat $S$-schemes.
Then $\alpha$ is $d$-birational if and only if for every $s \in S$, the restriction $\alpha_s: U_s \hookrightarrow X_s$ is $(d-\cod_S(s))$-birational.
\end{lemma}
\begin{proof}
Suppose $\alpha$ is $d$-birational.
Let $s \in S$ and $x \in X_s$ with $\cod_{X_s}(x) \le d-\cod_S(s)$.
By the codimension formula we deduce \[ \cod_X(x) = \cod_S(s) + \cod_{X_s}(x) \le d, \] and hence $x \in U \cap X_s = U_s$.

Conversely, suppose the fiberwise condition holds.
Let $x \in X$ with $f(x) = s$, and suppose that $\cod_X(x) \le d$.
Then by the codimension formula again we have \[ \cod_{X_s}(x) = \cod_X(x) - \cod_S(s) \le d - \cod_S(s). \]
It follows that $x \in U_s \subset U$.

This concludes the proof.
\end{proof}

For a (finite locally free) morphism $p: S' \to S$, we write $R_p$ for the Weil restriction functor \cite[Chapter 7]{bosch2012neron}.
\begin{proposition} \label{prop:d-birational-weil}
Let $p: S' \to S$ be finite locally free and $X \to S'$ smooth and quasi-projective.
Let $\alpha: U \hookrightarrow X$ be a $d$-birational open immersion.
Then $R_p(\alpha): R_p(U) \to R_p(X)$ is a $d$-birational open immersion.
\end{proposition}
\begin{proof}
\todo{this is a mess...}
Open immersions and smooth schemes are preserved under Weil restriction \cite[Proposition 7.6.2(i), Proposition 7.6.5(h)]{bosch2012neron}.
In particular $R_p(X) \to S$ is flat.
Using Lemma \ref{lemm:d-birational-fiberwise} it is thus enough to show that for $s \in S$, the restriction $R_p(\alpha)_s$ is $(d-\cod_S(s))$-birational.
Let $s' \in S_s$.
Since $\dim S_s = 0$, the codimension formula implies that $\cod_S(s) = \cod_{S'}(s')$, and hence Lemma \ref{lemm:d-birational-fiberwise} implies that $\alpha_s$ is $(d - \cod_S(s))$-birational.
Since Weil restriction commutes with base change \cite[Proposition A.5.2(1)]{conrad2015pseudo}, we may thus assume that $S = Spec(k)$ is the spectrum of a field.
Applying Corollary \ref{corr:d-bir-bc}(2), we may assume that $k$ is algebraically closed.

Writing $S'$ as a finite disjoint union of its connected components and using Corollary \ref{cor:d-bir-prod}, we may assume that $S'$ is a finite local $k$-scheme, and so in particular $S' \to S$ is a universal homeomorphism \cite[Tags 00J8 and 01S4]{stacks-project}.
We claim that the canonically induced square
\begin{equation*}
\begin{CD}
p^* R_p U @>{p^*R_p(\alpha)}>> p^* R_p X \\
@VVV            @VVV     \\
U         @>{\alpha}>> X
\end{CD}
\end{equation*}
is cartesian.
Indeed for a scheme $T$ over $S'$, maps into $p^* R_p X$ (over $S'$) are the same as maps $T \times_k S' \to X$.
We thus need to show that a map $T \times_k S' \to X$ factors through $U$ if and only if the composite $T \to T \times_k S' \to X$ does\NB{details?}, which follows from the fact that the first map is a homeomorphism.

By \cite[Lemma A.5.11]{conrad2015pseudo}, the map $p^* R_p X \to X$ is faithfully flat, and hence $p^*R_p(\alpha)$ is $d$-birational by Corollary \ref{corr:d-bir-bc}(1).
Since $X \to S$ is faithfully flat so is $p^*R_p X \to S$, and hence $R_p(\alpha)$ is $d$-birational by Corollary \ref{corr:d-bir-bc}(2).
\end{proof}

With this preparation out of the way, we come to our main topic, Hilbert schemes.
For $X \in \Sm_k$ (quasi-projective, say) we write $\Hilb^{\fr}(\A^n, X)$ for the (ind-smooth) ind-scheme representing the functor $h^{\mathrm{nfr},n}(X)$ of \cite[\S5.1.4]{EHKSY}.
We have \[ \Hilb^{\fr}(\A^n, X) = \coprod_{d \ge 0} \Hilb^{\fr}_d(\A^n, X), \] the decomposition by degree; each $\Hilb^{\fr}_d(\A^n, X)$ is smooth.
Finally \[ \Hilb^\fr(\A^\infty, X) = \colim_n \Hilb^{\fr}(\A^n, X). \]

\begin{lemma} \label{lemm:Hilb-pres-bir}
If $\alpha: U \to X \in \Sm_k$ is a $d$-birational open immersion of smooth quasi-projective $k$-schemes, then so is $\Hilb^{\fr}_m(\A^n, U) \to \Hilb^{\fr}_m(\A^n, X)$.
\end{lemma}
\begin{proof}
There are maps $I \xrightarrow{q} Z \xrightarrow{p} \Hilb^\flci_m(\A^n)$ with $q$ smooth, $p$ finite locally free, such that $\Hilb^{\fr}_m(\A^n, X) \wequi R_p(I \times X)$ \cite[\S5.1.4]{EHKSY}.
The result thus follows from Corollary \ref{cor:d-bir-prod} and Proposition \ref{prop:d-birational-weil} (using that $\Hilb^\flci_m(\A^n)$ is smooth, and hence flat).
\end{proof}

Recall that a $k$-scheme $X$ is rational if there exists a span of $0$-birational open immersions $X \leftarrow U \to \A^n_k$ for some $n$.
\begin{lemma} \label{lemm:Hilb-rational}
$\Hilb^{\fr}_d(\A^n, *)$ is rational.
\end{lemma}
\begin{proof}
We use the notation from the proof of Lemma \ref{lemm:Hilb-pres-bir}.
By construction, $I$ is a $\GL_n$-torsor over $Z$ (for the Zariski topology), and hence birational to $\GL_n \times Z$\NB{ref? proof: Let $\xi_1, \dots, \xi_n$ be the generic points. Let $U_i'$ be an open nbd of $\xi_i$ over which the torsor is trivial. Put $U_i = U_i' \setminus \cup_{j \ne i}\overline{\{\xi_j\}}$. Then the $\xi_i \in U_i$ and the $U_i$ are disjoint and open.}, whence birational to $\A^{n^2} \times Z$.
By construction of the Weil restriction \cite[Theorem 7.6.4]{bosch2012neron}, $R_{p}(\A^{n^2} \times Z)$ is locally on $\Hilb^\flci(\A^n)$ isomorphic to a product with an affine space.
Using Proposition \ref{prop:d-birational-weil}, it is thus enough to show that $\Hilb^\flci_d(\A^n)$ is rational.
It is well-known to be birational to $\mathrm{Sym}^d(\A^n)$ (see e.g. \cite[Lemma 4.28 and Theorem 4.36]{jelisiejew-thesis}), which is rational since $\A^n$ is \cite{mattuck1968field}.

This concludes the proof.
\end{proof}

\section{The birational localizations} \label{sec:birational}
Denote by $L_\bir^d \Spc(k)$ the (Bousfield) localization obtained by inverting $d$-birational open immersions of smooth $k$-schemes.
See \cite[\S5.4.4]{lurie-htt} for one account on localization of presentable $\infty$-categories.
Variants of these localizations have been considered previously, for example by Kahn-Sujata \cite{kahn2016birational} and Pelaez \cite{pelaez2014unstable}.

Since $f \times \id_X$ is a $d$-birational open immersion whenever $f$ is (see e.g. Corollary \ref{cor:d-bir-prod}), this is a symmetric monoidal localization (e.g. apply \cite[Proposition 6.16]{bachmann-norms} to $S$ the spectrum of an algebraically closed field and $\scr C = \FEt_S \wequi \Fin$\todo{surely there is a better reference?}).
By Zariski descent, the same localization is obtained by inverting $d$-birational open immersions between smooth quasi-projective (or affine) $k$-schemes.

Write $L_\bir^d \SHS(k)$ for the localization obtained by inverting maps of the form $\Sigma^{\infty+n}_{S^1} f_+$, with $n \in \Z$ and $f$ a $d$-birational open immersion.
Similarly write $L_\bir^d \Spc(k)_*$ for the localization at maps of the form $f_+$, and $L_\bir^d \Spc^\fr(k)$ for the localization at maps of the form $Ff$.
These are also symmetric monoidal localizations.

Recall that $\SHS(k)(d)$ is defined as the localizing subcategory generated by $\SHS(k) \wedge \Gmp{d}$.
The reflection into its right orthogonal is denoted by $s_{[0,d-1]}$.
\begin{lemma} \label{lemm:bir-basics}
Let $k$ be any field.
\begin{enumerate}
\item The functors \[ \Spc(k) \to \Spc(k)_* \to \SHS(k) \] and $F: \Spc(k)_* \to \Spc^\fr(k)$ preserve $L_\bir^d$-equivalences.
\item The forgetful functor $\Spc(k)_* \to \Spc(k)$ commutes with $L_\bir^d$.
\item $L_\bir^d$-equivalences in $\Spc(k)_*$ are stable under finite products.
\item The forgetful functor $U: \Spc^\fr(k) \to \Spc(k)_*$ commutes with $L_\bir^d$.
\end{enumerate}

If $k$ is perfect, then also the following hold.
\begin{enumerate}
\setcounter{enumi}{4}
\item A morphism $\alpha: E \to F \in \SHS(k)$ is an $L_\bir^d$-equivalence if and only if $\cof(\alpha) \in \SHS(k)(d+1)$.
  In fact the localizing subcategory generated by objects of the form $\cof(\alpha)$, where $\alpha$ is an $L_\bir^d$-equivalence, is $\SHS(k)(d+1)$.
\item For $E \in \SHS(k)$ we have $L_\bir^d E \wequi s_{[0,d]} E$.
\end{enumerate}
\end{lemma}
In the proof, we shall make use of the theory of localizations of presentable $\infty$-categories and strongly saturated classes of morphisms \cite[\S5.5.4, Definition 5.5.4.5]{lurie-htt}.
\begin{proof}
(1) is clear by construction, (3) is immediate from (2), and (6) from (5).\NB{Details?}

(2) By construction the functor detects $L_\bir^d$-local objects.
This implies that it is enough to show that it preserves the strongly saturated class of morphisms (in $\Spc(k)_*$) generated by $d$-birational open immersions of smooth schemes.
By \cite[Lemma 2.10]{bachmann-norms}\NB{Applied to $\Spc(k)_* \to \Spc(k)$!}, for this it is enough to show that if $f$ is such a map and $X \in \Sm_k$, then $f \coprod \id_X \coprod \id_*$ is also a $d$-birational open immersion.
This is clear.

(4) Using (2), it suffices to show that the functor $\Spc^\fr(k) \to \Spc(k)$ commutes with $L_\bir^d$. Let us write $h^\fr: \Spc(k) \rightarrow \Spc^\fr(k)$ for the left adjoint of the previous functor; it is characterized as a sifted-colimit preserving functor which on sends the (motivic localization of the) presheaf represented by smooth $k$-scheme $X$ to the (motivic localization of the) presheaf $h^{\fr}(X)$ classifying tangentially framed correspondnences (see \cite[Definition 2.3.4]{EHKSY} and below).

Arguing as in (2), we apply Lemma \cite[Lemma 2.10]{bachmann-norms} to $\Spc^\fr(k)$.
Using that $\Spc^\fr(k)$ is semiadditive \cite[Proposition 3.2.10(iii)]{EHKSY} and the localization $L_\bir^d$ on $\Spc(k)$ is symmetric monoidal, it is enough to prove that if $f: X \to Y$ is a $d$-birational open immersion of smooth quasi-projective $k$-schemes, then $h^\fr(f): h^\fr(X) \to h^\fr(Y)$ becomes an equivalence in $L_\bir^d \Spc(k)$.
By \cite[Corollary 2.3.25 and Theorem 5.1.8]{EHKSY}, it is enough to show that $\Hilb^\fr(\A^\infty, X) \to \Hilb^\fr(\A^\infty, Y)$ is an $L^d_\bir$-equivalence.
This follows from Lemma \ref{lemm:Hilb-pres-bir}.

(5) Using Lemma \ref{lemm:stable-loc} below, it suffices to show the ``in fact'' part.
Since $\alpha: X \times (\A^n \setminus 0) \to X \times \A^n$ is $(n-1)$-birational and $\cof(\alpha) \wequi X_+ \wedge T^n$, one inclusion is clear; for the other one it is enough to show that if $U \hookrightarrow X$ is a $d$-birational open immersion, then $X/U \in \SHS(k)(d+1)$.
This is well-known; we include the proof for the convenience of the reader.
Let $Z = X \setminus U$; we shall prove the claim by induction on $\dim Z$.
By generic smoothness \cite[Tag 0B8X]{stacks-project}, there exists a smooth dense open $U' \subset Z$; let $Z' = Z \setminus U'$.
The cofiber sequence \[ X \setminus Z'/U \setminus Z' \to X/U \to X/X \setminus Z' \] implies that it is enough to show that $X \setminus Z'/U \setminus Z', X/X \setminus Z' \in \SHS(k)(d+1)$.
For the former this follows from homotopy purity \cite[Theorem 3.2.23]{A1-homotopy-theory}, for the latter it holds by induction.

This concludes the proof.
\end{proof}

We made use of the following technical result, which is surely well-known.
\begin{lemma}\label{lemm:stable-loc} \todo{reference?!}
Let $\scr C$ be a stable presentable $\infty$-category, and $S$ a small set of morphisms in $\scr C$ closed under desuspension.
Write $S^0$ for the localizing subcategory of $\scr C$ generated by cofibers of morphisms in $S$.
Then the strong saturation of $S$ consists of those maps $\alpha$ with $\cof(\alpha) \in S^0$.
\end{lemma}
\begin{proof}
Let $A$ be a strongly saturated class of morphisms stable under desuspensions, and write $A^1$ for the set of objects $X \in \scr C$ such that $0 \to X \in A$.
Then $A^1$ is closed under colimits and desuspensions, whence a localizing subcategory.
Moreover by stability\NB{i.e. the cofiber sequences $(X \to X) \to (X \to Y) \to (0 \to Y/X)$ and $(0 \to \Sigma^{-1} Y/X) \to (X \to X) \to (X \to Y)$ in the arrow category}, $\alpha: X \to Y$ is an $A$-equivalence (i.e. in $A$) if and only if $0 \to \cof(\alpha)$ is an $A$-equivalence.
It follows that $A^0=A^1$, and also that $A \mapsto A^0$ is an inclusion-preserving bijection between strongly saturated classes stable under desuspension and localizing subcategories.
Consequently strongly saturated classes stable under desuspension containing $S$ are in bijection with localizing subcategories containing $S^0$.
The result now follows from the observation that if $S$ is closed under desuspension, then so is its strong saturation.
\end{proof}

As usual we denote by $\SH(k)^\eff \subset \SH(k)$ the localizing subcategory generated by $\sigma^\infty \SHS(k)$ and put $\SH(k)^\eff(d) = \Gmp{d} \wedge \SH(k)^\eff$; this is equivalently the localizing subcategory generated by $\sigma^\infty \SHS(k)(d)$.
We put $L_\bir = L_\bir^0$.

\section{Proof of Conjecture~\ref{conj:1}} \label{sec:main}
Recall from \cite[\S4]{hoyois2018localization} the framed presheaf $\Z \in \PSh_\Sigma(\Corr^\fr(k))$, and from \cite{EHKSY,EHKSY3} the presheaf $\FSyn^\fr \wequi h^\fr(*)$.
There is an evident ``degree'' map of framed presheaves $\FSyn^\fr \to \Z$, factoring in fact through an evident sub-presheaf $\N$.
\begin{theorem} \label{thm:main}
Let $k$ be any field.
\begin{enumerate}
\item The map $\FSyn^\fr \to \N$ is an $L_\bir$-equivalence in $\Spc^{\fr}(k)$.
\item The map \[ \Sigma^\infty_\fr \FSyn^\fr \to \Sigma^\infty_\fr \Z \] identifies with the canonical map \[ \1 \to s_0(\1) \in \SH(k). \]
\end{enumerate}
\end{theorem}
\begin{proof}
(1) By Lemma \ref{lemm:bir-basics}(2,4), it suffices to show that the underlying map of unpointed motivic spaces is an $L_\bir$-equivalence.
Since it is the coproduct of the maps $\FSyn^\fr_d \to *$, and $\FSyn^\fr_d \stackrel{L_\mot}{\wequi} \Hilb^{\fr}_d(\A^\infty, *)$ \cite[Corollary 2.3.25]{EHKSY}, it suffices to show that each $\Hilb^{\fr}_d(\A^n, *)$ is rational.
This is Lemma \ref{lemm:Hilb-rational}.

(2) By \cite[Lemma B.1]{bachmann-norms} and \cite[Lemma 20]{hoyois2018localization}, all terms are stable under essentially smooth base change, so we may assume that $k$ is perfect.
It follows from (1) and Lemma \ref{lemm:bir-basics}(6) that $\1 \wequi \Sigma^\infty_\fr \FSyn^\fr \to \Sigma^\infty_\fr \N$ induces an equivalence on $s_0$.
Moreover $\Sigma^\infty_\fr \N \wequi \Sigma^\infty_\fr \Z$, since $\Z$ is the group completion of $\N$. It thus remains to show that $\Sigma^\infty_\fr \Z$ is right orthogonal to $\SH(k)^\eff(1)$.  But now, $\ul{\pi}_*(\Sigma^\infty_\fr \Z)_0 = 0$ if $* \ne 0$ and $= \Z$ else, so $\ul{\pi}_*(\Sigma^\infty_\fr \Z)_{-1} = 0$, which implies what we want since homotopy sheaves are unramified.

This concludes the proof.
\end{proof}

\begin{proof} [Proof of Conjecture~\ref{conj:1}]
The unit map $u: \1 \to \KGL$ induces a map $s_0(u): s_0(\1) \to s_0(\KGL)$.
Theorem~\ref{thm:main}(2), Conjecture~\ref{conj:2} (to be proved in the next section) and Voevodsky's arguments from \cite{voevodsky293possible} (see also Example \ref{ex:KGL}) imply that $s_0(u)$ is an equivalence.
The zero slice of $\KGL$ is identified with higher Chow groups in \cite[Theorem 6.4.1]{levine2008homotopy}; whence the result.
\end{proof}
\begin{remark} \label{rem:cyc} Theorem~\ref{thm:main}(2) shows that if $E \in \SH(k)^\eff$ is provided with a map $\1 \to E$ such that $\ul{\pi}_i(E)_0 = 0$ for $i \ne 0$ and $\ul{\pi}_0(\1)_0 \to \ul{\pi}_0(E)_0$ is isomorphic to $\ul{GW} \to \Z$, then the induced maps $s_0(\1) \to s_0(E) \leftarrow E$ are both equivalences.
Indeed for the second equivalence we just need to show that $f_1 E \wequi 0$, which follows from $\ul{\pi}_i(E)_{-1} = 0$ (since it is the contraction of the zero sheaf), and for the first equivalence we need only verify that we get an isomorphism on $\ul{\pi}_i(\ph)_0$, which is now immediate from the theorem.

In principle, one can identify $s_0(\1)$ with the spectrum representing higher Chow groups by verifying that the latter do satisfy these properties. However, at some point it needs to be proved that the spectrum representing higher Chow groups (even granting its existence) is effective.
Voevodsky proves this (in characteristic zero) by studying the birational geometry of motivic Eilenberg MacLane spaces, and Levine deduces this from his homotopy coniveau tower theory.
We have no new arguments for this. In effect, our argument replaces the deduction of $s_0(\1)$ from $s_0(\KGL)$ in \cite[\S 10]{levine2008homotopy} involving the ``reverse cycle map."
\end{remark}

\section{Proof of Conjecture~\ref{conj:2}}
For $d<0$, we define $\SHS(k)(d) = \SHS(k)$, and $f_d = \id: \SHS(k) \to \SHS(k)(d)$. The next result is Conjecture~\ref{conj:2}. In Levine's approach \cite{levine2008homotopy}, he directly proves Corollary \ref{cor:final} below, which immediately implies the theorem.

\begin{theorem} \label{thm:levine}
Let $k$ be a perfect field.
Then \[ \omega^\infty(\SH(k)^\eff(d+1)) \subset \SHS(k)(d+1). \]
\end{theorem}
\begin{proof}
If $d < 0$ there is nothing to show.
Otherwise by Lemma \ref{lemm:bir-basics}(5), it suffices to show that if $f: X \to Y$ is a $d$-birational open immersion of smooth, quasi-projective $k$-schemes, then $\omega^\infty \Sigma^\infty_+ f \in \SHS(k)$ is an $L_\bir^d$-equivalence.
Using Lemma \ref{lemm:gc-technical} below, this follows from Lemma \ref{lemm:bir-basics}(1,3,4).
\end{proof}

The following technical result is a variant of \cite[Proposition 4.4]{voevodsky293possible}.
\begin{lemma} \label{lemm:gc-technical}
Let $\scr X \in \Spc^\fr(k)$ (where $k$ is a perfect field).
Then $\omega^\infty \Sigma^\infty_\fr \scr X \in \SHS(k)$ can be obtained as a colimit of $S^1$-desuspensions of suspension spectra of the form $\Sigma^\infty_{S^1} U\scr X^{\times n}$.
Moreover this expression is natural in $\scr X$.
\end{lemma}
\begin{proof}
Writing $\omega^\infty \Sigma^\infty_\fr \scr X$ as the colimit of desuspensions of its constituent spaces, we have \[ \omega^\infty \Sigma^\infty_\fr \scr X \wequi \colim_n \Sigma^{-n} \Sigma^\infty_{S^1} \Omega^\infty \Sigma^n \Sigma^\infty_\fr \scr X. \]
It is thus enough to prove that the spaces \[ \Omega^\infty \Sigma^\infty_\fr \Sigma^n \scr X \in \Spc(k)_* \] are of the desired form, say for $n \ge 1$.
We view $\scr X \in \PSh_\Sigma(\Corr^\fr(k))$.
It follows from the motivic recognition principle \cite[Theorem 3.5.14]{EHKSY} that \[ \Omega^\infty \Sigma^\infty_\fr \Sigma^n \scr X \wequi L_\mot^\gp \Sigma^n \scr X. \]
Writing $\Sigma^n \scr X$ as an iterated sifted colimit and using semiadditivity, we find that $\Sigma^n \scr X \wequi B^n \scr X$, where the bar construction $B^n$ is just applied sectionwise.
In particular, $\Sigma^n \scr X \wequi B^n \scr X$ is of the desired form.
Since the forgetful functor $\PSh_\Sigma(\Corr^\fr(k)) \to \PSh_\Sigma(\Sm_{k})$ preserves motivic equivalences \cite[Proposition 3.2.14]{EHKSY}, it suffices to show that $B^n \scr X$ is motivically group complete (i.e. $L_\mot B^n\scr X$ is group complete).
Since $n\ge 1$, $B^n \scr X$ is sectionwise connected (see e.g. \cite[Proposition 1.5]{segal1974categories}\NB{This reference is overkill (for a bisimplicial set with $X_{00} = *$, $|X_{\bullet\bullet}| \wequi diag X$ is connected since it has only one 0-simplex).}), and hence $L_\mot B^n \scr X$ is connected by \cite[Corollary 3.3.22]{A1-homotopy-theory}.
This concludes the proof.
\end{proof}

We deduce the following structural result, also originally due to Levine \cite[Theorem 9.0.3]{levine2008homotopy}.
Recall that $\SH(k) \wequi \SHS(k)[\Gmp{-1}]$ and so objects of $\SH(k)$ can be modeled by \emph{$\Gm$-$\Omega$-spectra}, i.e. sequences $(E_0, E_1, \dots)$ with $E_i \in \SHS(k)$, together with equivalences $E_i \wequi \Omega_\Gm E_{i+1}$.
\begin{corollary} \label{cor:final}
\begin{enumerate}
\item The functor $\omega^\infty: \SH(k) \to \SHS(k)$ commutes with the functors $f_d$ and $s_d$, for all $d \in \Z$.
\item If $E \in \SH(k)$ is represented by the $\Gm$-$\Omega$-spectrum $(E_0, E_1, \dots)$, then $f_d E$ is represented by the $\Gm$-$\Omega$-spectrum $(f_d E_0, f_{d+1} E_1, \dots)$.
  In particular $E$ is effective if and only if $E_i \in \SHS(k)(i)$.
\end{enumerate}
\end{corollary}
\begin{proof}
(1) The categories $\SH(k)^\eff(d)$ and $\SHS(k)(d)$ define non-negative parts of $t$-structures by \cite[Proposition 1.4.4.11]{lurie-ha}.
By construction $\sigma^\infty$ is right-$t$-exact, and hence $\omega^\infty$ is left-$t$-exact.
Since it is also right-$t$-exact by Theorem \ref{thm:levine}, we deduce that $\omega^\infty$ is $t$-exact, i.e. commutes with $f_d$.
The claim about $s_d$ follows immediately.

(2) We have $E_i \wequi \omega^\infty(\Gmp{i} \wedge E)$.
Hence we compute \[ f_d(E)_i \wequi \omega^\infty(\Gmp{i} \wedge f_d E) \wequi \omega^\infty(f_{d+i}(\Gmp{i} \wedge E)) \wequi f_{d+i}(\omega^\infty(\Gmp{i} \wedge E)) \wequi f_{d+i}(E_i). \]
Here we have used part (1) for the third equivalence (which is the only non-trivial one in the string above). This concludes the proof.
\end{proof}

We also deduce the following principle, slightly generalizing an argument of Voevodsky. The analog in topology is the following fact: if $E$ is a connective spectrum, then it is $d$-connective as soon as its infinite loop space is $d$-connective. 

\begin{proposition} \label{prop:detect-eff}
Let $k$ be a perfect field and $E \in \SH(k)^\veff$.
Then $E \in \SH(k)^\eff(d)$ as soon as $\Sigma^\infty \Omega^\infty E \in \SH(k)^\eff(d)$.
\end{proposition}
\begin{proof}
We may assume that $d \ge 0$.
Using the recognition principle and Lemma \ref{lemm:gc-technical}, or alternatively \cite[Proposition 4.4]{voevodsky293possible}, we find that $\omega^\infty E$ is in the localizing subcategory generated by suspension spectra of products of $\Omega^\infty E$.
Since \cite[Lemma 6.2.2 and Footnote 45]{morel-trieste} \[ \Sigma^\infty_{S^1}(X \times Y) \wequi \Sigma^\infty_{S^1}X \vee \Sigma^\infty_{S^1}Y \vee \Sigma^\infty_{S^1}X\wedge Y, \] this is equivalently the localizing subcategory generated by smash powers of $\Sigma^\infty_{S^1} \Omega^\infty E$.
It follows that $\sigma^\infty \omega^\infty E$ is in the localizing subcategory generated by smash powers of $\Sigma^\infty \Omega^\infty E$, and hence (1) $\sigma^\infty \omega^\infty E \in \SH(k)^\eff(d)$.
By the triangle identities, the composite \[ \omega^\infty E \to \omega^\infty \sigma^\infty \omega^\infty E \to \omega^\infty E \] is the identity, whence (2) $\omega^\infty E$ is a summand of $\omega^\infty \sigma^\infty \omega^\infty E$.

Corollary \ref{cor:final} implies that $F \in \SH(k)^\eff$ is $d$-effective if and only if $\omega^\infty F$ is $d$-effective.
Thus $\omega^\infty \sigma^\infty \omega^\infty E$ is $d$-effective by (1), whence so is $\omega^\infty E$ by (2), and hence so is $E$.

This concludes the proof.
\end{proof}

\begin{example} \label{ex:HZ}
We can use this to give a slightly different proof of Theorem \ref{thm:main} (i.e. determine $s_0(\1)$) over perfect fields.
Namely let $F$ denote the fiber of the degree map $\FSyn^{\fr,\gp} \to \Z$.
It is enough to show that $\Sigma^\infty F \in \SH(k)^\eff(1)$.
But by \cite{BEHKSY} we have \[ F \stackrel{L_\mot}{\wequi} \FSyn^{\fr+}_\infty \stackrel{L_\mot}{\wequi} \Hilb^{\fr}_\infty(\A^\infty,*)^+. \]
Since $\Sigma^\infty$ inverts acyclic maps, the result thus follows again from Lemma \ref{lemm:Hilb-rational}.
\end{example}

\begin{example} \label{ex:KGL}
The argument of Example \ref{ex:HZ} is modeled on Voevodsky's determination of $s_0(\KGL)$, which we can restate in our language as follows:
Using that $\Omega^\infty \KGL = \Z \times \Gr$, where $\Gr$ is the infinite Grassmannian variety, arguing as above one is reduced to showing that $\Gr$ is rational.
This is well-known.
\end{example}

\begin{example}
The converse of Proposition \ref{prop:detect-eff} is false.\todo{Slicker example?}
Let $E = H\Z/2 \wedge \Gm$.
We claim that $\Sigma^\infty \Omega^\infty E \not\in \SH(k)^\eff(1)$.
For this it suffices to construct a non-zero map $\Sigma^\infty \Omega^\infty E \to \Sigma H\Z/2$, or equivalently a non-zero map $\Omega^\infty H\Z/2 \wedge \Gm \to \Omega^\infty \Sigma H\Z/2$.
Any grouplike monoid $M$ is equivalent as a pointed space to $\pi_0(M) \times M_{\ge 1}$\NB{ref?}; applying this construction sectionwise and projecting to the $M_{\ge 1}$ part, we obtain the desired non-zero map \[ \Omega^\infty H\Z/2 \wedge \Gm \to \Omega^\infty (H\Z/2 \wedge \Gm)_{\ge 1} \wequi K(\Z/2, 1) \wequi \Omega^\infty \Sigma H\Z/2. \]
\end{example}

\bibliographystyle{alpha}
\bibliography{bibliography}

\end{document}